\documentclass[twoside,11pt]{article}
\usepackage{amsmath, amsthm, amscd, amsfonts, amssymb, graphicx, color}
\usepackage{graphicx}
\begin{document}
\setcounter{page}{1}
\setlength{\unitlength}{10mm}
\newcommand{\f}{\frac}
\newtheorem{theorem}{Theorem}[section]
\newtheorem{lemma}[theorem]{Lemma}
\newtheorem{proposition}[theorem]{Proposition}
\newtheorem{corollary}[theorem]{Corollary}
\theoremstyle{definition}
\newtheorem{definition}[theorem]{Definition}
\newtheorem{example}[theorem]{Example}
\newtheorem{solution}[theorem]{Solution}
\newtheorem{xca}[theorem]{Exercise}
\theoremstyle{remark}
\newtheorem{remark}[theorem]{Remark}
\numberwithin{equation}{section}
\newcommand{\cardinal}[1]{\vert #1 \vert}
\newcommand{\Bset}{\left\{ 0 , 1 \right\}}
\newcommand{\srough}{$(U,W,T,S)$}
\newcommand{\srougheq}{$G=(U,W,T,S)$}
\newcommand{\gup}[1]{\overline{G}(#1)}
\newcommand{\gupp}[1]{\overline{G'}(#1)}
\newcommand{\glo}[1]{\underline{G}(#1)}
\newcommand{\eref}[1]{(\ref{#1})}
\newcommand{\sta}{\stackrel}
\title{A combinatorial approach to certain topological spaces based on minimum complement S-approximation spaces}
\author { M. R. Hooshmandasl\footnote{Speaker}, M. Alambardar Meybodi, A. K. Goharshady and A. Shakiba}
\date{}
\maketitle
\begin{picture}(5,2)(3,-9)
 \put(0.5,-1.5){\includegraphics[width=2.5cm]{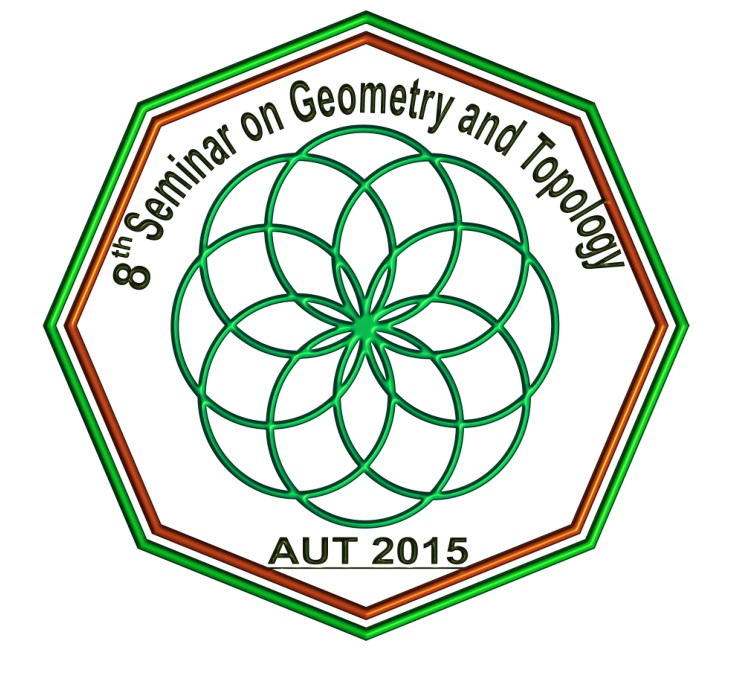}}
  \put(6.3,-0.0){ \textcolor{black}{\textbf{$8^{th}$ Seminar on Geometry and Topology} }}
 \put(5.3,-0.7){\textcolor{blue}{\emph{ Amirkabir University of Technology, December 15-17, 2015}}}
 \put(16.5,-1.5){\includegraphics[width=2.3cm]{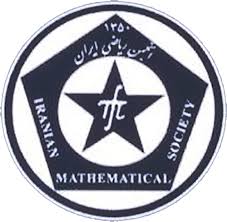}}
\end{picture}
\begin{abstract}
\noindent   An S-approximation space is a novel approach to study systems with uncertainty that are not expressible in terms of inclusion relations. In this work, we further examined these spaces, mostly from a topological point of view by a combinatorial approach. This work also identifies a subclass of these approximation spaces, called $S_\mathcal{MC}$-approximations. Topological properties of this subclass are investigated and finally, the topologies formed by $S_\mathcal{MC}$-approximations are enumerated up to homeomorphism.

\vspace{.5cm}\leftline{{AMS subject Classification 2010: 54F99, 68P01
}}

\vspace{.5cm}\leftline{Keywords: S-approximation space, $S_\mathcal{MC}$-approximation, Rough set, Combinatorial Enumeration.}
\end{abstract}

\section{Introduction}
In almost all real-life applications we should handle uncertainty. In non-crisp sets, uncertainty is characterized by boundary regions, non-empty subsets of the universe where nothing can be said about their element memberships. Approximation is one of the ways to deal with these uncertainties. In classical set theory, a subset $A$ of a universe $U$ induces a partition $\left\{ A, U - A \right\}$ on that universe. This partition might be interpreted as a knowledge about elements of $U$, i.e. elements of $A$ are indiscernible. The same thing holds for elements of $U-A$. This can be generalized to any partition $\mathcal{P}$ of $U$, supposing that elements in the same equivalence class of $\mathcal{P}$ are not distinguishable but those in different classes are. In consequence, for a subset $A$ of $U$, the problem of whether $x$ belongs to $A$ or not, with respect to knowledge $\mathcal{P}$, may become undecidable, i.e. we may have indiscernible elements, with respect to $\mathcal{P}$, which are or are not members of $A$. To cope with such uncertainty, a number of tools are invented such as Dempster-Shafer theory of evidence \cite{Shafer1976}, theory of fuzzy sets \cite{Zadeh1986,Zadeh1984,Zadeh1983,Zadeh1965}, and theory of rough sets \cite{Pawlak1991,Pawlak1984,Pawlak1982}. Rough set theory and Dempster-Shafer theory of evidence are two independent approaches for handling uncertainty, but there is an important resemblance between the two. More formally, lower and upper approximations of rough set theory correspond to the inner and outer reductions from Dempster-Shafer theory\cite{Grzymala1988}.

Since their introduction in 1980s \cite{Pawlak1982}, rough sets have been applied to many different areas, such as discovering data patterns, a core subject of data mining, and dealing with incomplete information systems\cite{Pawlak2002, Pawlak2007}. Studying rough set and its generalizations together with topology has been an interesting research topic, as discussed in \cite{Polkowski2002, L1, L2} and the connection between rough set theory and topology was found early in the framework of topology of partitions. Basic dependence of rough sets on certain topological spaces made this discovery not too unexpected\cite{Mousavi2001, Pei2011}. L. Polkowski implemented topological spaces using rough sets which were based on information systems \cite{Polkowski2002}. A. Skowron in $1988$ and A. Wiweger at the same time, but independently, discussed it on Z. Pawlak's rough sets. The relationship between the modified sets, topological spaces and rough sets based on pre-order was considered by J. Kortelainen in $1994$\cite{Kortelainen1994}. This discussion was continued in \cite{Lin1992, Lin1998}. Analyzing the relation between generalized rough sets and topologies from different viewpoints is another interesting research area.

Rough set theory and its generalizations are all based on the inclusion relation \cite{Pawlak.1988, Pawlak.1994, Pawlak1982, Pawlak1984, Yao1996, Yao2003, Yao1996a, Yao1996b}, which can be considered  as a limitation. In this work, we use a new concept named $S$-approximation set which is proposed in \cite{Hooshmand.Unpub}. This concept is independent from the inclusion relation and contains rough sets and their generalizations as special cases. It is also applied in three-way decision theory in \cite{Shakiba2015}, neighborhood systems in \cite{kais}, fuzzy and intuitionistic fuzzy set theories \cite{jais}.

Then we will study topological spaces built  upon these approximations. Moreover, we will discuss homeomorphisms between such topologies and state the necessary and sufficient condition under which two such topologies are homeomorphic. It is well-known that homeomorphism is an equivalence relation over the class of all topologies. We will count the number of equivalence classes under homeomorphism over these topologies as well.

\section{Preliminaries}	
In this section, we introduce the fundamental concepts of rough sets, generalizations of rough sets and topological spaces.
\subsection{Basic Rough Set and its Generalizations}\label{sec.rough}
Let $U$ be a non-empty finite set and $R \subseteq U \times U$, an equivalence relation on $U$. This relation partitions the set $U$ into equivalence classes like $\left[ x \right]_{R}$ which consists of all $y \in U$ such that $x R y$. Let $X$ be a subset of $U$, then the set $X$ can be approximated by equivalence classes of $R$ by constructing the lower and upper approximations of $X$ with respect to $R$, as is proposed by Z. Pawlak in \cite{Pawlak1982} as follows
$$ \underline{app}_R(X) = \left\{ x \in U \vert \left[ x \right]_{R} \subseteq X \right\} , $$
$$ \overline{app}_R(X) = \left\{ x \in U \vert \left[ x \right]_{R} \cap X \neq \emptyset \right\} . $$

If $\underline{app}_R(X) = \overline{app}_R(X)$, then the set $X$ is called \emph{definable} with respect to $R$, otherwise it is called a \emph{rough set} with respect to $R$. The ordered pair $(\underline{app}_R(X), \overline{app}_R(X))$ is called the \emph{approximation space} of $X$ with respect to $R$.
\begin{proposition}[\cite{Pawlak1982}]
	\label{prop.propertiesOfRoughSets}
	Let $U$ be a non-empty finite set and $R \subseteq U \times U$ denotes an equivalence relation on $U$, then for every $X, Y \subseteq U$ the following properties hold:
	\begin{enumerate}
		\item $\underline{app}_R(X) \subseteq X \subseteq \overline{app}_R(X)$,
		\item $\underline{app}_R(U) = \overline{app}_R(U) = U$ and $\underline{app}_R(\emptyset) = \overline{app}_R(\emptyset) = \emptyset$,
		\item $\overline{app}_R(X \cup Y) = \overline{app}_R(X) \cup \overline{app}_R(Y)$,
		\item $\underline{app}_R(X \cap Y) = \underline{app}_R(X) \cap \underline{app}_R(Y)$,
		\item $X \subseteq Y$ implies that $\underline{app}_R(X) \subseteq \underline{app}_R(Y)$,
		\item $X \subseteq Y$ implies that $\overline{app}_R(X) \subseteq \overline{app}_R(Y)$,
		\item $\underline{app}_R(X) \cup \underline{app}_R(Y) \subseteq \underline{app}_R(X \cup Y)$,
		\item $\overline{app}_R(X \cap Y) \subseteq \overline{app}_R(X) \cap \overline{app}_R(Y)$,
		\item $\underline{app}_{R}(X) = (\overline{app}_R(X^c))^c$ and equivalently $\overline{app}_{R}(X) = (\underline{app}_R(X^c))^c$.
	\end{enumerate}
\end{proposition}

Y. Yao's extension of Z. Pawlak's rough set is obtained by using an arbitrary relation, possibly not equivalence \cite{Yao1996}. Let $R$ be a binary relation on $U$. Then the ordered pair $(U,R)$ is called a \emph{generalized approximation space} based on the relation $R$. For $X \subseteq U$, the lower and upper approximations of set $X$ are generalized as
$$  \underline{app}_R(X) = \left\{ x \in U \vert R(X) \subseteq X \right\}, $$
and
$$  \overline{app}_R(X) = \left\{ x \in U \vert R(X) \cap X \neq \emptyset \right\}, $$
where $R(x) = \left\{ y \in U \vert (x,y) \in R \right\}$.

\begin{proposition} \cite{Yao1996} Let $U$ be a non-empty finite set and $R \subseteq U \times U$ an arbitrary  relation on $U$, then for every $X, Y \subseteq U$, properties of proposition \eref{prop.propertiesOfRoughSets} are satisfied.
\end{proposition}

There are also many other generalizations of rough set theory such as rough set models for incomplete information systems \cite{Skowron1996, Slowinski1996, Krys1998, Liang2002}, rough set models based on coverings \cite{Zakowski1983, Boni1998, Zhu2003} and rough fuzzy sets or fuzzy rough sets \cite{Dub1990}. Moreover, these models can be generalized to the case of two universes \cite{Wong1995,Wong1992} based on the Shafer's compatibility view \cite{Shafer1987, Pei2004}, generalized rough fuzzy sets \cite{Wu2003}, and arbitrary relations \cite{Davvaz.2008}.

\subsection{Topological Spaces}
In this section, we would briefly review basic concepts of topologies. A pair $(X, \tau)$ where $X$ is a non-empty set and $\tau$ is a family of subsets of $X$ containing $\emptyset$ and $X$ is called a \emph{topology} if $\tau$ is closed under arbitrary union and finite intersection. Members of $\tau$ are called open sets and their complements are called closed sets. Let $(X, \tau_X)$ and $(Y, \tau_Y)$ be two topologies. A function $f: X \rightarrow Y$ is said to be \emph{continuous} if for any open subset $A$ in $Y$, $f^{-1}(A)$ is also open in $X$. Moreover, a bijective continuous map $\Phi : X \rightarrow Y$ is called a \emph{homeomorphism} if $\Phi^{-1}$ is also continuous. If such a $\Phi$ exists, then $(X, \tau_X)$ and $(Y, \tau_Y)$ are called \emph{homeomorphic} topologies. Homeomorphic topologies form equivalence classes over any set of topologies.

For every binary relation $R$ over $U$, we can examine a topology generated by $R$. The \emph{right neighborhood} is defined as $xR = \left\{ y \in U \vert (x,y) \in R \right\}$, and the topology over $U$ is denoted by $(U, \tau_R)$, where $\tau_R = \left\{ xR \vert x \in U \right\}$. For more details, refer to \cite{Li2012, Lashin2005}.

\subsection{$S$-approximation}\label{sec.signedRoughSets}
$S$-approximation is a new mathematical approach to study approximation spaces \cite{Hooshmand.Unpub}. This approach is proposed on the basis of the ideas of Dempster's multi-valued mappings \cite{Dempster1967}, and has the Pawlak's rough set and its generalizations as special cases. These spaces are first proposed in \cite{Hooshmand.Unpub} and are reviewed in this section.
\begin{definition}[\cite{Hooshmand.Unpub}]
	\label{grs.def.grs}
	An $S$-approximation is the quadruple $G = (U, W, T, S)$ where $U$ and $W$ are finite non-empty sets, $T$ is a mapping of the form $T: U \rightarrow P^{\star}(W)$ and $S$ is a mapping of the form $S : P^{\star}(W) \times P^{\star}(W) \rightarrow \left\{ 0 , 1 \right\}$.
	
	For a non-empty subset $X$ of $W$, the upper and lower approximations of $X$ are defined as follows:
	$$ \overline{G}(X) = \left\{ x \in U \vert S(T(x), X^c) = 0 \right\} , $$
	and
	$$ \underline{G}(X) = \left\{ x \in U \vert S(T(x), X) = 1 \right\} , $$
	where $X^c$ is the complement of $X$ with respect to $W$.
\end{definition}

\subsubsection{$S_{\mathcal{M}}$-approximations}\label{sec.minConditionedSrough}

There exists a know sub-class of $S$-approximation spaces which satisfy properties $(3)$ to $(10)$ of proposition \eref{prop.propertiesOfRoughSets}, but are not inclusion-based. These properties are satisfied because their $S$ relation satisfies the $S$-min condition, introduced in \cite{Hooshmand.Unpub}.

\begin{definition}[$S$-min Condition \cite{Hooshmand.Unpub}]
	Let $G = (U,W,T,S)$ be an $S$-approximation. We say that the relation $S : P^{\star}(W) \times P^{\star}(W) \rightarrow \left\{ 0 , 1 \right\}$ is a relation in $S_{\mathcal{M}}$ class if it satisfies
	$$ S(A, B \cap C) = \min \left\{ S(A, B) , S(A, C) \right\} , $$
	for arbitrary non-empty subsets $A$, $B$, and $C$ of $W$. We also say an $S$-approximation $G=(U, W, T, S') $ is an $S_{\mathcal{M}}$-approximation  if $S'$ belongs to the $S_{\mathcal{M}}$ class.
\end{definition}
\begin{remark}
	The inclusion relation does indeed belong to the $S_{\mathcal{M}}$ class but there are other non-inclusion relations in this class as well, cf. \cite{Hooshmand.Unpub}.
\end{remark}

The following proposition is the counterpart of proposition \eref{prop.propertiesOfRoughSets} in $S_{\mathcal{M}}$-approximation spaces.
\begin{proposition}{\cite{Hooshmand.Unpub}}
	\label{s.app.prop}
	Let $G = (U,W,T,S)$ be an $S_{\mathcal{M}}$-approximation. For all $A, B \subseteq W$ and $x \in U$, the following hold:
	\begin{enumerate}
		\item $ A \subseteq B $ implies that for all $X \subseteq W$,  $S(X, B^c) \leq S(X, A^c) ,$
		\item $ \max\{  S(T(x), A),S(T(x), B) \}\leq S(T(x), A \cup B) ,$
		\item $\overline{G}(A \cup B) =  \overline{G}(A) \cup \overline{G}(B) $,
		\item $\underline{G}(A \cap B) = \underline{G}(A) \cap \underline{G}(B) $,
		\item $ A \subseteq B$ implies $\underline{G}(A) \subseteq \underline{G}(B) $,
		\item $ A \subseteq B$ implies $\overline{G}(A) \subseteq \overline{G}(B) $,
		\item $\underline{G}(A) \cup \underline{G}(B) \subseteq \underline{G}(A \cup B) $,
		\item $\overline{G}(A \cap B) \subseteq \overline{G}(A) \cap \overline{G}(B) $,
		\item $\underline{G}(A) = (\overline{G}(A^c))^c$ and equivalently $\overline{G}(A) = (\underline{G}(A^c))^c$.
	\end{enumerate}
\end{proposition}

Interestingly, it is not always the case that $\underline{G}(A) \subseteq \overline{G}(A)$ in $S_{\mathcal{M}}$-approximation spaces, although this property always holds in Pawlak's rough sets \cite{Hooshmand.Unpub}.
\begin{example}{\cite{Hooshmand.Unpub}}
	Suppose $G=(U,W,T,S)$ is an $S_{\mathcal{M}}$-approximation where
	$$ S(A,B) =  \left\{
	\begin{matrix}
	1 & \ & A \cup B = W \\
	0 & & \text{otherwise}
	\end{matrix}
	\right. ,
	$$
	$U = \left\{ a\right\}$, and $T(a) = W = \left\{ 1, 2 \right\}$.
	
	In this case $\overline{G}(\left\{ 1 \right\}) = \left\{ x \in U \vert T(x) \cup \left\{ 2 \right\} \neq  W \right\} = \emptyset$, while $$\underline{G}(\left\{ 1 \right\}) = \left\{ x \in U \vert T(x) \cup \left\{ 1 \right\} = W \right\} = \left\{a\right\},$$ so $\underline{G}(\left\{ 1 \right\}) \not\subseteq \overline{G}(\left\{ 1 \right\})$.
\end{example}
The structure of $S_{\mathcal{M}}$-approximations plays an important role in understanding the topological structures which will be introduced in later sections, so we remind some results from \cite{Hooshmand.Unpub}.
\begin{definition}[\cite{Hooshmand.Unpub}]
	\label{def.f.min}
	Let $W$ be a non-empty finite set. A function $f : P^{\star}(W) \rightarrow \left\{ 0 , 1 \right\}$ is said to be minimizing if for each $A, B \subseteq W$,
	$$ f(A \cap B) = \min \left\{ f(A), f(B) \right\} . $$
\end{definition}
\begin{lemma}{\cite{Hooshmand.Unpub}}
	\label{lem.subset.atom}
	Let $f : P^{\star}(W) \rightarrow \left\{ 0 , 1 \right\}$ be a minimizing function. For each $A,B \subseteq W$, if $A \subseteq B$, then $f(A) \leq f(B)$.
\end{lemma}
\begin{lemma}{\cite{Hooshmand.Unpub}}
	\label{lem.minPiecewise}
	Let $G=(U, W, T, S)$ be an $S_{\mathcal{M}}$-approximation and $\vert W \vert = n$. We label the non-empty subsets of $W$ as $ \left\{A_1 , \ldots , A_{2^n - 1}\right\}$. Then there exist minimizing functions $\left\{ f_1, \ldots, f_{2^n - 1} \right\}$ of the form $f_i : P^{\star}(W) \rightarrow \left\{ 0 , 1 \right\}$ such that for every $B \subseteq W$, we have $S(A_i,B) = f_i(B)$ for $1 \leq i \leq n$.
\end{lemma}

Lemma \eref{lem.minPiecewise} leads us towards counting the number and finding the structure of minimizing $f$s.

\begin{definition}[\cite{Hooshmand.Unpub}]
	\label{def.atom}
	Let $f : P^{\star}(W) \rightarrow \left\{ 0 , 1 \right\}$ be a minimizing function. A non-empty subset $\omega$ of the set $W$ is called an atom of $f$ if and only if $f(\omega) = 1$ and for each proper non-empty subset of $\omega$ such as $\eta$, $f(\eta)=0$.
\end{definition}
\begin{proposition}{\cite{Hooshmand.Unpub}}
	\label{prop.disjointAtoms}
	Let $f : P^{\star}(W) \rightarrow \left\{ 0 , 1 \right\}$ be a minimizing function and $\omega_1$ and $\omega_2$ two non-identical atoms of $f$. Then $\omega_1 \cap \omega_2 = \emptyset$.
\end{proposition}

\begin{proposition}{\cite{Hooshmand.Unpub}}
	\label{prop.structure}
	Let $f : P^{\star}(W) \rightarrow \left\{ 0 , 1 \right\}$ be a minimizing function and $\Upsilon$ the set of all atoms of $f$. Then for a subset $X$ of $W$, $f(X) = 1$ if and only if there exists $\omega \in \Upsilon$ such that $\omega \subseteq X$.
\end{proposition}

\begin{proposition}{\cite{Hooshmand.Unpub}}
	Let $f : P^{\star}(W) \rightarrow \left\{ 0 , 1 \right\}$ be a minimizing function, $\Upsilon$ the set of all atoms of $f$ and $\vert \Upsilon \vert \geq 2$. Then for each $x \in W$, $\left\{ x \right\}$ is an atom of $f$.
\end{proposition}

By previous propositions, it is clear that we either have no atoms, or exactly one atom or an atom per element.

\section{Topologies of $S_{\mathcal{MC}}$-approximations}\label{sec.topol}
In this paper, we are interested in topological structures over a special class of $S_{\mathcal{M}}$-approximations, where the $S$ relation is extended to $P(W) \times P(W) \rightarrow \{0,1\}$ and satisfies the complement condition, which is defined as:
\begin{equation}
\label{eq.s.complement.condition.lemma}
S(A, B^c) = 1 - S(A, B) ,
\end{equation}
for any $A, B \subseteq W$ and $S(A, \emptyset) = 0$. We use the notation $S_\mathcal{MC}$ to denote this class of $S_\mathcal{M}$-approximations.
\begin{lemma}
	\label{lem.s.w.s.empty}
	Let \srougheq $ $ be an $S_{\mathcal{MC}}$-approximation then for each $A \subseteq W$,
	\begin{align}
	\begin{aligned}
	S(A, W) = & 1 ,
	\end{aligned}
	\end{align}
	and
	\begin{align}
	\begin{aligned}
	\underline{G}(A) = \overline{G}(A) .
	\end{aligned}
	\end{align}
\end{lemma}
\begin{proof}
	By $S$-complement condition, we have $S(A, W) = 1 - S(A, \emptyset) = 1$.
	
	For the last part we have
	\begin{align}
	\begin{aligned}
	\gup{A} = & \left\{ x \in U \vert S(T(x) , A^c) = 0 \right\} \\
	= & \left\{ x \in U \vert S(T(x) , A) = 1 \right\} \\
	= & \glo{A} .
	\end{aligned}
	\end{align}			
\end{proof}

Assume $G=$\srough $ $ is an $S_{\mathcal{MC}}$-approximation, then if we define $\tau$ as $\left\{ \gup{ A } \vert A \subseteq W \right\}$,  $(U,\tau)$ becomes a topology. This claim is stated more precisely in the following theorem.
\begin{theorem}
	\label{the.s.rough.gup.topology}
	Let \srougheq $ $ be an $S_{\mathcal{MC}}$-approximation, and $\tau$ be defined as
	\begin{equation}
	\label{eq.tau.gup}
	\tau = \left\{ \gup{ A } \vert A \subseteq W \right\} .
	\end{equation}
	Then $(U, \tau)$ is a topology.
\end{theorem}
\begin{proof}
	According to definition of topology, $(U, \tau)$ should satisfy three conditions.
	\begin{enumerate}
		\item We claim that $\gup{W} = U$ and $\gup{\emptyset} = \emptyset$, so $U$ and $\emptyset$ belong to $\tau$. By definition of $\gup{X}$, we have
		\begin{align}\label{eq.proof.topol.1.1}\begin{aligned}
		\gup{W} = & \left\{ x \in U \vert S(T(x), W^c) = 0 \right\} \\
		= & \left\{ x \in U \vert S(T(x), \emptyset) = 0 \right\} \\
		= & U .
		\end{aligned}\end{align}
		and
		\begin{align}\label{eq.proof.topol.1.2}\begin{aligned}
		\gup{\emptyset} = & \left\{ x \in U \vert S(T(x), \emptyset^c) = 0 \right\} \\
		= & \left\{ x \in U \vert S(T(x), W) = 0 \right\} \\
		= & \emptyset .
		\end{aligned}\end{align}
		\item It can be easily seen from the definition of $\tau$ that for each $Y_i \in \tau$, where $i$ is in some index set $\mathcal{I}$, there exists $A_i \subseteq W$ such that $Y_i = \gup{A_i}$. So, it is the case that
		\begin{align}\label{eq.property.2.gup}\begin{aligned}
		\cup_{i \in \mathcal{I}} Y_i = & \cup_{i \in \mathcal{I}} \gup{A_i} \\
		= & \cup_{i \in \mathcal{I}} \left\{ x \in U \vert S(T(x), A_i^c) = 0 \right\} \\							
		= & \left\{ x \in U \vert \vee_{i \in \mathcal{I}} \left( S(T(x), A_i^c) = 0 \right) \right\} \\
		= & \left\{ x\in U \vert \min_{i \in \mathcal{I}} \left\{ S(T(x), A_i^c) \right\} = 0 \right\} \\
		= & \left\{ x\in U \vert S(T(x), \cap_{i \in \mathcal{I}} A_i^c) = 0 \right\} \\
		= & \gup{\cup_{i \in \mathcal{I}} A_i} . 
		\end{aligned}\end{align}
		The latter equality is obtained by theorem \eref{s.app.prop}, property $(1)$. Therefore $\cup_{i \in \mathcal{I}} \gup{A_i} \in \tau$. Note that $U$ and $W$ are assumed to be finite.
		\item
		We have
		\begin{align}
		\begin{aligned}
		x \in \cap_{i=1}^n \gup{A_i} \Leftrightarrow & \forall i\in \{1, \ldots , n\}~~ \:\: S(T(x), A_i^c) = 0 \\
		\Leftrightarrow & \forall i \in \{1, \ldots , n\}~~ \:\: S(T(x), A_i) = 1 \\
		\Leftrightarrow & S(T(x), \cap_{i=1}^n A_i) = 1 \\
		\Leftrightarrow & S(T(x), (\cap_{i=1}^n A_i)^c) = 0 \\
		\Leftrightarrow & x \in \gup{\cap_{i=1}^n A_i} .
		\end{aligned}
		\end{align}
		Therefore $\cap_{i=1}^n \gup{A_i} \in \tau$.
	\end{enumerate}
\end{proof}
If $S$ belongs to $S_{\mathcal{MC}}$, then $\gup{A} = \glo{A}$ for every $A \subseteq W$. From theorems \eref{lem.s.w.s.empty} and \eref{the.s.rough.gup.topology}, the following corollary is obtained, i.e. $(U, \tau)$ is also a topology when $\tau = \left\{ \glo{A} \vert A \subseteq W \right\}$.
\begin{corollary}
	\label{cor.glo.topol}
	Let \srougheq $ $ be an $S_{\mathcal{MC}}$-approximation. Define $\tau = \left\{ \glo{A} \vert A \subseteq W \right\}$, then $(U, \tau)$ is the same topology as in theorem \eref{the.s.rough.gup.topology}.
\end{corollary}
The set $\tau$ has the property that it is closed under complement.
\begin{theorem}
	Let $(U, \tau)$ be a topology obtained by theorem \eref{the.s.rough.gup.topology}, then $(U,\tau)$ is a \emph{clopen} topology, i.e. every open set is closed.
\end{theorem}
\begin{proof}
	Let $A \subset W$, then by \eref{s.app.prop}, $\glo{A} = (\gup{A^c})^c$, and since by theorem \eref{lem.s.w.s.empty}, $\gup{A}  = \glo{A}$, so $\gup{A}$ is closed.
	$\gup{W} = U$ and $\emptyset = U^c \in \tau$, therefore $\gup{W}$ is also closed.
	
	Let $Y$ be a non-empty closed subset of $U$, so $Y=(\gup{A})^c$ for some $A \subset W$, and since $(\gup{A})^c = \glo{A^c} = \gup{A^c}$, $Y$ is open. It is obvious that $\emptyset$ is open. This concludes the proof.
\end{proof}

\section{Enumerating $S$ Functions in $S_{\mathcal{MC}}$}\label{sec.numberTopol}
In this section, we suppose that $U$, $W$, and $T:U \rightarrow \mathcal{P}^*(W)$ are fixed and then we enumerate all the functions  $S$ where $G=(U,W,T,S)$ is an $S_{\mathcal{MC}}$.

Let $W$ be a non-empty finite set, and $f:P(W) \rightarrow \Bset$ a minimizing function with a single atom, denoted by $\alpha(f)$. If $\cardinal{\alpha(f)} = 1$, then its only element is denoted by ${\boldsymbol{a}}(f)$. The following theorems state the effect of the $S$-min condition, and a much stricter version of it, on $S$ in terms of its atoms and tend to be very useful tools in counting non-homeomorphic topologies, as we will discuss later.
\begin{theorem}
	\label{the.single.atom}
	Let $W$ be a non-empty finite set, and $f:P(W) \rightarrow \Bset$ a minimizing function. Then the following are equivalent:
	\begin{enumerate}
		\item For every non-empty subset $A$ of $W$, \begin{align}
		\label{eq.the.complement.leq}
		\begin{aligned}
		f(A^c) \leq 1 - f(A) .
		\end{aligned}
		\end{align}
		\item Either $f$ has a single atom or $f \equiv 0$, i.e. $f$ has no atoms.
	\end{enumerate}
\end{theorem}
\begin{proof}
	\begin{description}
		\item[$(1 \rightarrow 2)$] The proof is by contradiction. Suppose that $f$ does not satisfy $(2)$, then $\vert W \vert \geq 2$ and each unary subset of $W$ is an atom of $f$. This way, $f(A) = 1$ for every non-empty subset of $W$ and $f(\emptyset) = 0$. Suppose $A$ is a non-empty proper subset of $W$. Since $A$ has at least one element and every unary set of $W$ is an atom, then $f(A) = 1$. On the other hand, $A^c$ is non-empty and for the same reason, $f(A^c) = 1$ which contradicts $(1)$.
		\item[$(2 \rightarrow 1)$] Suppose $f(A) = 1$ for some non-empty $A \subseteq W$. So $\alpha(f) \subseteq A$, which implies $\alpha(f) \not\subseteq A^c$, hence $f(A^c) =0$. This obviously yields to $(1)$.
	\end{description}
\end{proof}

\begin{theorem}
	\label{the.single.atom.single.member}
	Let $W$ be a non-empty finite set, and $f:P(W) \rightarrow \Bset$ a minimizing function. Then the following are equivalent:
	\begin{enumerate}
		\item For every subset $A$ of $W$, \begin{align}
		\label{eq.the.complement.eq}
		\begin{aligned}
		f(A^c) = 1 - f(A) .
		\end{aligned}
		\end{align}
		\item $f$ has a single atom and $\vert \alpha(f) \vert = 1$.
	\end{enumerate}
\end{theorem}
\begin{proof}
	\begin{description}
		\item[$(1 \rightarrow 2)$] By lemma \eref{lem.s.w.s.empty}, $f \not\equiv 0$, so by theorem \eref{the.single.atom}, $f$ has a single atom $\alpha(f)$. Now we should show that $\vert \alpha(f) \vert = 1$. The proof is by contradiction. Suppose that $\vert \alpha(f) \vert > 1$ and ${\boldsymbol{\alpha}}_1 \in \alpha(f)$. Then $f(\left\{ {\boldsymbol{\alpha}}_1 \right\}) = 0$. On the other hand, $f(\left\{ {\boldsymbol{\alpha}}_1 \right\}^c) = 0$ since ${\boldsymbol{\alpha}}_1 \in \alpha(f)$. This is a contradiction with $(1)$.
		\item[$(2 \rightarrow 1)$] Recall that if $\vert \alpha(f) \vert = 1$, then ${\boldsymbol{a}}(f)$ denotes its only element. Suppose $f(A) = 1$ for some $A \subseteq W$. So ${\boldsymbol{a}}(f) \in A$, which implies ${\boldsymbol{a}}(f) \not\in A^c$, hence $f(A^c) =0$. This statement can be reversed, so $(1)$ holds.
	\end{description}
\end{proof}

Theorem \eref{the.single.atom.single.member} makes it easy to count the number of different functions $S$ in $S_\mathcal{MC}$ that can be used to define topologies as stated in theorem \eref{the.s.rough.gup.topology}, this number is clearly an upper-bound for the number of different topologies that can be formed as in that theorem.
\begin{theorem}
	\label{the.num.srough.topologies}
	Let \srougheq be an $S$-approximation, and fix the sets $U$, $W$, and the relation $T$. Then the number of different $S$ functions that can be used in order for $G$ to be in the $S_\mathcal{MC}$ equals $\cardinal W^{2^{\cardinal W}-1}$.
\end{theorem}
\begin{proof}
	This number can be obtained using the multiplication principle since there are exactly $\cardinal W$ minimizing functions $f : P(W) \rightarrow \Bset$ that have a single atom $\alpha(f)$ such that $\vert \alpha(f) \vert = 1$.
\end{proof}
\begin{remark}
	It is notable that the number obtained above is an upper bound on the number of distinct topologies with fixed $U$, $W$, and $T$.
\end{remark}

\section{Non-Homeomorphic Topologies of $S_{\mathcal{MC}}$-approximations}\label{sec.hom.topol}
In this section, we would first establish a necessary and sufficient condition so that two topologies $(U, \tau)$ and $(U', \tau^{\prime})$ generated by two $S_\mathcal{MC}$-approximations \srougheq and $G'=(U', W', T', S')$ respectively as in theorem \eref{the.s.rough.gup.topology} such that $\cardinal{U} = \cardinal{U'}$ and $\cardinal{W} = \cardinal{W'}$, are homeomorphic. Then we will use this condition to count such non-homeomorphic topologies.
\begin{lemma}
	\label{lem.homo.same.size}
	Let $(U,\tau)$ and $(U',\tau^{\prime})$ be two homeomorphic topologies, $\Phi : U \rightarrow U' $, a homeomorphism between them, $u \in U$, $u' = \Phi(u)$, and $A \in \tau$ an open set containing $u$, then $\Phi(A)$ contains $u'$ and has the same cardinality as $A$.
\end{lemma}
\begin{proof}
	It is straightforward.
\end{proof}
For the sake of easier stating the proof of theorem \eref{the.hom.topol}, we introduce the notion of degree for each element of $W$.
\begin{definition}
	\label{def.degree}
	Let $(U, \tau)$ be a topology as in theorem \eref{the.s.rough.gup.topology}, then the degree of $w \in W$ is defined as
	\begin{equation}
	\label{eq.def.degree.element}
	\deg_G(w) = \vert \left\{ u \in U \vert \alpha(f_{T(u)}) = \left\{ w \right\} \right\} \vert .
	\end{equation}
	Also, the set $W_i$, where $i$ is a non-negative integer, is defined as
	\begin{equation}
	\label{eq.w.i.g.set}
	W_i = \left\{ w \in W \vert \deg_G (w) = i \right\} .
	\end{equation}
\end{definition}
\begin{lemma}
	The set of $W_i$'s, as defined in definition \eref{def.degree}, is a partition of $W$.
\end{lemma}
\begin{lemma}
	\label{lem.deg.g.eq.card.gup}
	Let \srougheq $ $ be an $S_\mathcal{MC}$-approximation that forms a topology as in theorem \eref{the.s.rough.gup.topology}. Then for each $w \in W$, $\deg_G(w) = \cardinal{\gup{\left\{ w \right\}}}$.
\end{lemma}
\begin{proof}
	It is sufficient to show that $\gup{\left\{ w \right\}} = \left\{ x \in U \vert {\boldsymbol{a}}(f_{T(x)}) = w \right\}$.
	\begin{align}
	\begin{aligned}
	\gup{\left\{ w \right\}} = & \left\{ x \in U \vert S(T(x), \left\{ w \right\}) = 1 \right\} \\
	= & \left\{ x \in U \vert \alpha(f_{T(x)}) \subseteq \left\{ w \right\} \right\} \\
	= & \left\{ x \in U \vert {\boldsymbol{a}}(f_{T(x)}) = w \right\} .
	\end{aligned}
	\end{align}
\end{proof}
\begin{theorem}
	\label{the.hom.topol}
	Let \srougheq $ $ and $G' = (U',W',T',S')$ be two $S_\mathcal{MC}$-approximations such that $\vert U \vert = \vert U' \vert$ and $\vert W \vert = \vert W' \vert$, then their corresponding topologies that are formed as in theorem \eref{the.s.rough.gup.topology} are homeomorphic if and only if for each non-negative integer $i$, $ \vert W_i \vert = \vert W_i^{\prime} \vert$.
\end{theorem}
\begin{proof}
	Suppose $\vert W_i \vert = \vert W_i^{\prime} \vert$, for all $i$. We define a function of the form $\gamma : W \rightarrow W'$ such that for each $w_i \in W$ and $w_j^{\prime} \in W'$, $\gamma(w_i) = w_j^{\prime}$ implies $\deg_G (w_i) = \deg_{G'} (w^{\prime}_j)$. Since $\cardinal{W_i} = \cardinal{W_i^{\prime}}$ for all $i$, we can define a bijective function of this kind. So, from now on, we assume that $\gamma$ is one-to-one and onto, i.e. a bijection.
	
	Now we define a function $\Phi : U \rightarrow U'$ such that for each $u_i \in U$ and $u_j^{\prime} \in U'$, $\Phi(u_i) = u_j$ implies that $\gamma({\boldsymbol{a}}(f_{T(u_i)})) = {\boldsymbol{a}}(f^{\prime}_{T'(u_j^{\prime})})$, where $f^{\prime}_{A'}(B') = S'(A',B')$ for every $A',B' \subseteq W'$. By definition of $\gamma$ it is obvious that $\Phi$ can be defined to be a bijection, since
	\begin{align}
	\begin{aligned}
	\deg_{G'} ({\boldsymbol{a}}(f^{\prime}_{T'(u_j^{\prime})})) = & \deg_{G'} (\gamma ( {\boldsymbol{a}}(f_{T(u_i)})) ) \\
	= & \deg_G ({\boldsymbol{a}}(f_{T(u_i)})) .
	\end{aligned}
	\end{align}
	Assuming it so, we show that $\Phi$ is a homeomorphism between $(U, \tau)$ and $(U', \tau^{\prime})$, where $(U', \tau^{\prime})$ is the topology formed by $G'$ as in theorem \eref{the.s.rough.gup.topology}.
	
	It is sufficient to show that $\Phi^{-1}$ is continuous, continuity of $\Phi$ can be proved in a similar manner.
	
	Let $A \subseteq W$ and $H = \gup{A}$. We need to show that there exists $A' \subseteq W'$ such that $\Phi(H) = \gupp{A'}$.
	\begin{align}
	\begin{aligned}
	H = \gup{A} = & \left\{ u \in U \vert S(T(u), A) = 1 \right\} \\
	= &  \left\{ u \in U \vert \alpha(f_{T(u)}) \subseteq A \right\} \\
	= & \left\{ u \in U \vert {\boldsymbol{a}}(f_{T(u)}) \in A \right\} .
	\end{aligned}
	\end{align}
	We define $A'$ as follows,
	\begin{align}
	\begin{aligned}
	A' = & \left\{ {\boldsymbol{a}}(f^{\prime}_{T'(u')}) \vert u' \in \Phi(H) \right\} \\
	= & \left\{ \gamma ( {\boldsymbol{a}}(f_{T(u)})) \vert u \in H \right\} .
	\end{aligned}
	\end{align}
	It is clear by definition of $\gamma$ and $\Phi$ that $A' = \gamma(A)$.
	
	We have
	\begin{align}	
	\begin{aligned}
	\gupp{A'} = & \left\{ u' \in U' \vert S'(T'(u'), A') = 1 \right\} \text{ (* by definition of } \gupp{\cdot} \text{ *) }\\
	= & \left\{ u' \in U' \vert \alpha(f^{\prime}_{T'(u')}) \subseteq A' \right\} \text{ (* all } f \text{'s are single-atomic *)} \\
	= & \left\{ u' \in U' \vert {\boldsymbol{a}}(f^{\prime}_{T'(u')}) \in A' \right\} \text{ (* all atoms are unary *)} \\
	= & \left\{ u' \in U' \vert \exists u \in H \:\: {\boldsymbol{a}}(f^{\prime}_{T'(u')}) = \gamma({\boldsymbol{a}}(f_{T(u)}))  \right\} \text{ (* by definition of } A' \text{ *)} \\
	= & \left\{ u' \in U' \vert \exists u \in H \:\: u' = \Phi(u) \right\} \text{ (* by definition and bijectiveness of } \Phi \text{ *)} \\
	= & \Phi(H) .
	\end{aligned}
	\end{align}
	So for every open $H$ in $(U,\tau)$, $\Phi(H)$ is also open in $(U', \tau^{\prime})$ which means that $\Phi^{-1}$ is continuous.
	
	Conversely, let $\Phi$ be some arbitrary homeomorphism between $(U, \tau)$ and $(U', \tau^{\prime})$. We prove that $\cardinal{W_i} = \cardinal{W_i^{\prime}}$, for all positive integers $i$. It is clear that in this case, for $i=0$, $\cardinal{W_0}$ would be equal to $\cardinal{W_0^{\prime}}$ if the equality holds for all other $i$'s, since $\cardinal{W} = \cardinal{W'}$, because $W_i$'s partitions $W$ and $W_i^{\prime}$'s partition $W'$.
	
	Let $w \in W \cup W'$ be an element with minimal positive degree. We can assume that $w \in W$ without any loss of generality, since $\Phi^{-1}$ is a homeomorphism as well. Let $u$ be such an element of $U$ that $\alpha(f_{T(u)})=\{w\}$. By property (4) proposition \eref{s.app.prop}, $\gup{\alpha(f_{T(u)})}$ is the smallest open set containing $u$. Let $u' = \Phi(u)$, so $\cardinal{\gup{\alpha(f_{T(u)})}} = \cardinal{\gupp{\alpha(f_{T'(u')}^{\prime})}}$, because the former is the smallest open set containing $u$ and the latter is the smallest open set containing $u'$ and their size must be equal according to lemma \eref{lem.homo.same.size}. It must be the case that all elements of $\Phi(\gup{\alpha(f_{T(u)})})$ have the same single element atom, because $\gup{\alpha(f_{T(u)}}$ is a minimal non-empty open set, and so $\Phi(\gup{\alpha(f_{T(u)})})$ is also a minimal non-empty open set. Let's name the element of this atom as $w'$, we claim that $\deg_{G'}(w') = \deg_{G}(w)$, and this happens since all elements of $\Phi(\gup{\alpha(f_{T(u)})})$ share one single atom. By lemma \eref{lem.deg.g.eq.card.gup}, 
	\begin{align}
	\begin{aligned}
	\deg_G(w) = & \cardinal{\gup{\left\{ w \right\}}} = \cardinal{\gup{\alpha(f_{T(u)})}} \\
	= & \cardinal{\gupp{\alpha(f^{\prime}_{T'(u')}}} = \cardinal{\gupp{\left\{ w' \right\}}} \\
	= & \deg_{G'}(w') .
	\end{aligned}
	\end{align}
	So, $\Phi$ maps all elements of $\gup{\left\{ w \right\}}$ to all elements of $\gupp{\left\{ w' \right\}}$.
	
	Let's define $G_1 = (U_1 = U - \gup{\left\{ w \right\}}, W_1 = W - \left\{ w \right\}, T_1, S_1)$ where $T_1$ and $S_1$ are induced by $U_1$ and $W_1$ from $T$ and $S$ in $G$, respectively. Let's define $G_1^{\prime} = \left( U_1^{\prime} = U' - \overline{G}(\left\{ w'\right\}), W_1^{\prime} = W' - \left\{ w' \right\}, T_1^{\prime}, S_1^{\prime} \right)$ similarly and $\Phi_1$ as the induced version of $\Phi$ by $U_1$ and $U_1 ^{\prime}$. It is easy to verify that $\Phi_1$ is a homeomorphism between $G_1$ and $G_1^{\prime}$. Continuing with the same procedure, according to finite descent principle leads us to the desired result, since $W$ and $U$ are finite sets.
\end{proof}
\begin{lemma}
	\label{lem.cardinal.u.based.deg.w}
	Let $G$ be defined as in theorem \eref{lem.s.w.s.empty}, then
	\begin{equation}
	\sum_{w \in W} \deg_G(w) = \cardinal{U} .
	\end{equation}
\end{lemma}
\begin{proof}
	It is straightforward.
\end{proof}
Let $M$ be a universe. Then we define $\mathcal{T}_{m,n}$ as the set of all topologies $(U, \tau)$ made by some \srougheq $ $ as in theorem \eref{lem.s.w.s.empty}, such that $U, W \subseteq M$, $\cardinal{U} = m$, and $\cardinal{W} = n$. In the following theorem, we will count the number of equivalence classes of $\mathcal{T}_{m,n}$ by $p(m,n)$, where $p(m,n)$ denotes the number of unordered partitions of $m$ into a maximum of $n$ positive integer summands, or equivalently the number of unordered partitions of $m$ into exactly $n$ non-negative integers\cite{Niven1965}.
\begin{theorem}
	\label{the.non.homo.count}
	The number of equivalence classes of $\mathcal{T}_{m,n}$ under homeomorphism is $p(m,n)$.
\end{theorem}
\begin{proof}
	According to theorem \eref{the.hom.topol}, it is sufficient to show that there exists some bijection between partitions of $m$ into $n$ non-negative summands and possible combinations of $\cardinal{W_i}$'s, since every combination of $W_i$'s corresponds to a unique equivalence class under homeomorphism. In compliance with lemma \eref{lem.cardinal.u.based.deg.w}, we can construct such a bijection by defining $\cardinal{W_i} =$ number of repetitions of the integer $i$ in the partition. It is easily seen that this method defines a bijection.
\end{proof}

\section{Conclusion}
$S$-approximation is a novel tool for studying approximation of uncertain data which is not necessarily described by inclusion relation.
We identified a sub-class of $S$-approximations, called $S_{\mathcal{MC}}$-approximations that have certain topological characteristics that lead to existence of some topologies which are investigated in this paper along with some of their properties. Finally, we have enumerated these topologies up to homeomorphism.

\date{ \small M. R. Hooshmandasl \\
Department of Mathematics and Computer Sciences\\
Yazd University, Yazd, Iran \\
The Laboratory of Quantum Information Processing, Yazd University, Yazd, Iran\\
E-mail:hooshmandasl@yazd.ac.ir}

\date{ \small Alambardar Meybodi \\
Department of Mathematics and Computer Sciences\\
Yazd University, Yazd, Iran \\
The Laboratory of Quantum Information Processing, Yazd University, Yazd, Iran\\
E-mail:alambardar@stu.yazd.ac.ir}

\date{ \small A. K. Goharshady \\
	Institute of Science and Technology Austria,\\ Klosterneuburg, Austria.\\
	E-mail:goharshady@ist.ac.at}

\date{ \small A. Shakiba \\
	Department of Mathematics and Computer Sciences\\
	Yazd University, Yazd, Iran \\
	The Laboratory of Quantum Information Processing, Yazd University, Yazd, Iran\\
	E-mail:ali.shakiba@stu.yazd.ac.ir}


\begin{thebibliography}{99}

\bibitem{Boni1998}
Z. Bonikowski, E. Bryniariski, and U. Skardowska, \emph{Extension and
	intensions in the rough set theory}, Information Sciences 107 (1998), pp.
149--167.

\bibitem{Davvaz.2008}
B. Davvaz, \emph{A short note on algebraic $t$-rough sets}, Information
Sciences 178 (2008), pp. 3247--3252.

\bibitem{Dempster1967}
A. Dempster, \emph{Upper and lower probabilities induced by a multivalued
	mapping}, The annals of mathematical statistics 38 (1967), pp. 325--339.

\bibitem{Dub1990}
D. Dubois and H. Prade, \emph{Rough fuzzy sets and fuzzy rough sets},
International Journal of General Systems 17 (1990), pp. 191--209.

\bibitem{Grzymala1988}
J. Grzymala-Busse, \emph{Knowledge acquisition under uncertainty -- a rough set
	approach}, Journal of intelligent and Robotic Systems 1 (1988), pp. 3--16.

\bibitem{Hooshmand.Unpub}
M. Hooshmandasl, A. Shakiba, A. Goharshady, and A. Karimi,
\emph{S-approximation: a new approach to algebraic approximation}, Journal of
Discrete Mathematics  (2014).

\bibitem{Zhu2003}
F. Zhu and F. Wang, \emph{Reduction and axiomatization of covering generalized
	rough sets}, Information Sciences 152 (2003), pp. 217--230.


\bibitem{Kortelainen1994}
J. Kortelainen, \emph{On relationship between modified sets, topological spaces
	and rough sets}, Fuzzy Sets and Systems 61 (1994), pp. 91--95.

\bibitem{Krys1998}
M. Kryszkiewicz, \emph{Rough set approach to incomplete information systems},
Information Sciences 112 (1998), pp. 39--49.

\bibitem{L1}
E. Lashin, A. Kozae, A.A. Khadra, and T. Medhat, \emph{Rough set theory for
	topological spaces}, International Journal of Approximate Reasoning 40
(2005), pp. 35--43.

\bibitem{Lashin2005}
E. Lashin, A. Kozae, A.A. Khadra, and T. Medhat, \emph{Rough set theory for
	topological spaces}, International Journal of Approximate Reasoning 40
(2005), pp. 35--43.

\bibitem{Li2012}
Z. Li, T. Xie, and Q. Li, \emph{Topological structure of generalized rough
	sets}, Computers and Mathematics with Applications 63 (2012), pp. 1066--1071.

\bibitem{Liang2002}
J. Liang and Z. Xu, \emph{The algorithm on knowledge reduction in incomplete
	information systems}, International Journal of Uncertainty, Fuzziness, and
Knowledge-based Systems 10 (2002), pp. 95--103.

\bibitem{Lin1992}
T. Lin, \emph{Topological and fuzzy rough sets}, in \emph{Intelligent Decision
	Support}, R. Slowinski, ed., Kluwer,  1992, pp. 287--304.

\bibitem{Lin1998}
T. Lin, \emph{Granular computing on binary relations (i)}, in \emph{Rough Sets
	in Knowledge Discovery}, L. Polkowski and A. Skowron, eds., Vol.~1,
Physica-Verlag,  1998, pp. 107--121.

\bibitem{Mousavi2001}
A. Mousavi and P. Maralani, \emph{Relative sets and rough sets}, International
Journal of Applied Mathematics and Computer Science 11 (2001), pp. 637--345.

\bibitem{Niven1965}
I. Niven, \emph{Mathematics of Choice: How to count without counting},
Mathematical Association of America (MAA), 1965.

\bibitem{Pawlak1982}
Z. Pawlak, \emph{Rough sets}, International Journal of Computer and Information
Sciences 11 (1982), pp. 341--356.

\bibitem{Pawlak1984}
Z. Pawlak, \emph{Rough classification}, International Journal of Man-Machine
Studies 20 (1984), pp. 469--483.

\bibitem{Pawlak1991}
Z. Pawlak, \emph{Rough Sets: Theoretical Aspects of Reasoning About Data},
Theory and decision library: System theory, knowledge engineering, and
problem solving, Kluwer Academic Publishers, 1991.

\bibitem{Pawlak2002}
Z. Pawlak, \emph{Rough sets and intelligent data analysis}, Information
Sciences 147 (2002), pp. 1--12.

\bibitem{Pawlak.1994}
Z. Pawlak and A. Skowron, \emph{Rough membership functions}, in \emph{Advances
	in the Dempster-Shafer Theory of Evidence}, R. Yager, M. Fedrizzi, and J.
Kacprzyk, eds., Wiley,  1994, pp. 251--271.

\bibitem{Pawlak2007}
Z. Pawlak and A. Skowron, \emph{Rough sets and boolean reasoning}, Information
Sciences 177 (2007), pp. 41--73.

\bibitem{Pawlak.1988}
Z. Pawlak, S. Wong, and W. Ziarko, \emph{Rough sets: Probabilistic versus
	deterministic approach}, International Journal of Man-Machine Studies 29
(1988), pp. 81--95.

\bibitem{L2}
Z. Pei, D. Pei, and L. Zheng, \emph{Topology vs generalized rough sets},
International Journal of Approximate Reasoning 52 (2011), pp. 231--239.

\bibitem{Pei2011}
Z. Pei, D. Pei, and L. Zheng, \emph{Topology vs generalized rough sets},
International Journal of Approximate Reasoning 52 (2011), pp. 231--239.

\bibitem{Pei2004}
Z. Pei and Z. Xu, \emph{Rough set models on two universes}, International
Journal of General Systems 33 (2004), pp. 569--581.

\bibitem{Polkowski2002}
L. Polkowski, \emph{Rough sets: Mathematical foundations}, Vol.~15, Physica
Verlag, 2002.

\bibitem{Shafer1976}
G. Shafer, \emph{A mathematical theory of evidence}, Vol.~1, Princeton
university press Princeton, 1976.

\bibitem{Shafer1987}
G. Shafer, \emph{Belief functions and possibility measures}, in \emph{Analysis
	of Fuzzy Information}, J. Bezdek, ed., Vol.~1, CRC Press,  1987, pp. 51--84.

\bibitem{Shakiba2015}
A. Shakiba, and M. R. Hooshmandasl \emph{S-approximation Spaces: A Three-way Decision Approach}, Fundamenta Informaticae 139(3) (2015), pp. 307--328.

\bibitem{kais}
A. Shakiba, and M. R. Hooshmandasl \emph{Neighborhood system S-approximation spaces and applications}, Knowledge and Information Systems (2015), pp. 1--46, doi: 10.1007/s10115-015-0913-9.

\bibitem{jais}
A. Shakiba, M. R. Hooshmandasl, B. Davvaz and S. A. Shahzadeh Fazeli \emph{An Intuitionistic Fuzzy Approach to S-approximation Spaces}, Journal of Intelligent and Fuzzy Systems (2016), pp. 1--12 (to appear).

\bibitem{Skowron1996}
A. Skowron and J. Stepaniuk, \emph{Tolerance approximation spaces}, Fundamenta
Informaticae 27 (1996), pp. 245--253.

\bibitem{Slowinski1996}
R. Slowinski and J. Stefanowski, \emph{Rough-set reasoning about uncertain
	data}, Fundamenta Informaticae 27 (1996), pp. 229--243.

\bibitem{Wong1995}
S. Wong, L. Wang, and Y. Yao, \emph{On modeling uncertainty with interval
	structures}, Computational Intelligence 11 (1995), pp. 406--426.

\bibitem{Wong1992}
S.M. Wong, L. Wang, and Y. Yao, \emph{Interval structure: a framework for
	representing uncertain information}, in \emph{Proceedings of the Eighth
	international conference on Uncertainty in artificial intelligence}, 1992,
pp. 336--343.

\bibitem{Wu2003}
W. Wu, J. Mi, and W. Zhang, \emph{Generalized fuzzy rough sets}, Information
Sciences 151 (2003), pp. 263--282.

\bibitem{Yao1996a}
Y. Yao, \emph{Two views of the theory of rough sets in finite universes},
International Journal of Approximate Reasoning 15 (1996), pp. 291--317.

\bibitem{Yao2003}
Y. Yao, \emph{On generalizing rough set theory}, in \emph{Rough Sets, Fuzzy
	Sets, Data Mining, and Granular Computing}, G. Wang, Q. Liu, Y. Yao, and A.
Skowron, eds., Lecture Notes in Computer Science, Vol. 2639, Springer Berlin
Heidelberg,  2003, pp. 44--51.

\bibitem{Yao1996}
Y. Yao and T. Lin, \emph{Generalization of rough sets using modal logic},
Intelligent Automation and Soft Computing 2 (1996), pp. 103--120.

\bibitem{Yao1996b}
Y. Yao, S. Wong, and T. Lin, \emph{A review of rough set models}, in
\emph{Rough sets and data mining}, Springer,  1996, pp. 47--75.

\bibitem{Zadeh1965}
L. Zadeh, \emph{Fuzzy sets}, Information and control 8 (1965), pp. 338--353.

\bibitem{Zadeh1983}
L. Zadeh, \emph{The role of fuzzy logic in the management of uncertainty in
	expert systems}, Fuzzy sets and Systems 11 (1983), pp. 197--198.

\bibitem{Zadeh1984}
L. Zadeh, \emph{A computational theory of dispositions}, in \emph{Proceedings
	of the 10\textsuperscript{th} international conference on Computational
	linguistics}, 1984, pp. 312--318.

\bibitem{Zadeh1986}
L. Zadeh, \emph{A simple view of the dempster-shafer theory of evidence and its
	implication for the rule of combination}, AI magazine 7 (1986), p.~85.

\bibitem{Zakowski1983}
W. Zakowski, \emph{Axiomatization in the space $(u, \pi)$}, Demonstratio
Mathematica XVI (1983), pp. 761--769.
    
\end{thebibliography}
\end{document}